\documentclass[12pt, twoside]{article}
\usepackage{amsmath,amsthm,amssymb}
\usepackage{times}
\usepackage{enumerate}
\usepackage{bm}
\usepackage[all]{xy}
\usepackage{extarrows}
\usepackage{color}
\def\titlerunning#1{\gdef\titrun{#1}}
\makeatletter
\def\author#1{\gdef\autrun{\def\and{\unskip, }#1}\gdef\@author{#1}}
\def\address#1{{\def\and{\\\hspace*{18pt}}\renewcommand{\thefootnote}{}%
\footnote {#1}}%
\markboth{\autrun}{\titrun}}
\makeatother
\def\email#1{e-mail: #1}
\def\subjclass#1{{\renewcommand{\thefootnote}{}%
\footnote{\emph{Mathematics Subject Classification (2010):} #1}}}
\def\keywords#1{\par\medskip
\noindent\textbf{Keywords.} #1}

\newtheorem{thm}{Theorem}[section]

\newtheorem{lem}[thm]{Lemma}
\newtheorem{prop}[thm]{Proposition}
\theoremstyle{definition}
\newtheorem{defin}[thm]{Definition}
\newtheorem{rem}[thm]{Remark}

\numberwithin{equation}{section}

\input xy
\xyoption{all}

\def \C {\mathbb{C}}

\def \a {\alpha }
\def \b {\beta}

\def \La {\Lambda}
\def\w {\omega}
\def\Om{\Omega}

\def\pa{\partial}
\def\na {\nabla}

\begin{document}

\baselineskip=17pt

\titlerunning{Stable Yang-Mills connections on Special Holonomy Manifolds}

\title{Stable Yang-Mills connections on Special Holonomy Manifolds}

\author{Teng Huang}

\date{}

\maketitle

\address{T.~Huang: Department of Mathematics, University of Science and Technology of China,
                   Hefei, Anhui 230026, PR China; \email{oula143@mail.ustc.edu.cn}}


\begin{abstract}
We prove that energy minimizing Yang-Mills connections on a compact $G_{2}$-manifold has holonomy equal to $G_{2}$ are $G_{2}$-instantons,\ subject to an extra condition on the curvature.\ Furthermore,\ we show that energy minimizing connections on a compact Calabi-Yau $3$-fold has holonomy equal to $SU(3)$ subject to a similar condition are holomorphic.


\end{abstract}

\section{Introduction}

Let $G$ be a compact Lie group and $E$ a principal $G$-bundle on a complete oriented Riemannian manifold $M$.\ Let $A$ denote a connection on $E$ and
$\na_{A}$ the associated covariant derivative on the adjoint bundle $\mathfrak{g}_{E}$.\ The Yang-Mills energy of $A$ is
$$ YM(A):=\|F_{A}\|^{2}_{L^{2}(M)}$$
where $F_{A}$ denotes the curvature of $A$.\ A connections is called a Yang-Mills connection if it is a critical point of the Yang-Mills functional.

It is well known that,\ on a $4$-manifold,\ a special class of solutions exists that are automatically global minimizers of this functional.\ These are the self-dual or anti-self-dual connections,\ also called instantons.\ The fact that they are global minimizers of the Yang-Mills functional is a consequence of an energy identity that relates the $L^{2}$-norms of the self-dual or anti-self-dual components of $F_{A}$ to topological data (characteristic classes) of the bundle.\ Not all Yang-Mills connections are instantons.\ See \cite{SS,SSU} for example of $SU(2)$ Yang-Mills connection on $S^{4}$ which are neither self-dual nor anti-self-dual.

It is an interesting question to consider,\ rather than global minimizers of $YM$,\ the local minimizers.\ These are also known as stable critical points,\ characterized by the fact that the second variation of $YM$ at such a critical point is nonnegative.\ Earlier work by Bourguignon and Lawson \cite{bourguignon1981stability} used variational techniques to argue that if $M$ was a compact,\ homogeneous $4$-manifold,\ and the structure group of the vector bundle $E$ was $SU(2)$ or $SU(3)$,\ then a stable Yang-Mills connection is either an instanton or abelian (a direct sum of connections on $U(1)$ bundles).

In \cite{MS},\ Stern considered the minimizing Yang-Mills connections on compact homogeneous $4$-manifold,\ he proved that those connections are ether
instantons or split into a sum of instantons on passage to the adjoint bundle.

In higher dimensions,\ the instanton equation on $M$ can be introduced as follows.\ Assuming there is a closed $(n-4)$-form $\Om$ on $M$ (For our
purposes,\ $M$ is a $G_{2}$ manifold and $\Om$ is the fundamental 3-form $\phi^{(3)}$ or $M$ is a Calabi-Yau $3$-fold and $\Om$ is the $(1,1)$-form $\w$).\ A connection,\ $A$\ ,\ is called an anti-self-dual instanton,\ if it satisfies the instanton equation
\begin{equation}\nonumber
F_{A}=-\ast(\Om\wedge F_{A}).
\end{equation}
Instantons on the higher dimension,\ proposed in \cite{CDFJ} and studied in \cite{DT,DS,HT,RC,G.Tian,RSW},\ are important both in mathematics
\cite{DS,G.Tian} and string theory \cite{GSW}.

In higher dimensional,\ it is interesting to study when a Yang-Mills connection is an instanton.\ In this paper we consider the stable Yang-Mills connections over compact $G_{2}$-manifolds and compact Calabi-Yau $3$-folds.

On $G_{2}$-manifolds,\ the $2$-forms decompose as
$$\La^{2}(M)=\La^{2}_{7}(M)+\La^{2}_{14}(M).$$
where the fiber of $\La^{2}_{k}$ is an irreducible $G_{2}$ representation of dimension $k$.\ Let $F_{A}=F^{7}_{A}+F^{14}_{A}$ be the corresponding
decomposition of the curvature.\ Then we call a connection,\ $A$,\ {a $G_{2}$-instanton},\ if $F^{7}_{A}=0$ (see \cite{DT,RC} ).\ Then,\ we {prove} that energy minimizing connections on a compact $G_{2}$-manifold are in fact $G_{2}$-instantons,\ subject to an extra condition on the curvature.
\begin{thm}(Main Theorem)\label{I1}
Let $(M,\phi^{(3)})$ be a compact $G_{2}$-manifold {which} has holonomy equal to $G_{2}$,\ let $A$ be a stable Yang-Mills connection on a bundle $E$ over $M$ with compact,\ semi-simple Lie group.\ Assume $\psi_{A}$ is $d_{A}$-harmonic,\ $\psi_{A}\triangleq\frac{1}{3}(\ast(F_{A}^{7}\wedge\psi^{(4)})$,\ then $A$ is a $G_{2}$-instanton.
\end{thm}
Our proof of Theorem \ref{I1} extends the idea of Stern \cite{MS}.\ Let $A(t)$ be a smooth family of connections on $E$ with $A(0)=A$.\ The assumption that $A$ is a local minimum of the Yang-Mills energy implies the variational inequality
$$\frac{d^{2}}{dt^{2}}\mid_{t=0}YM(A(t))\geq0.$$
The proof of the theorem relies on choosing useful families of test {connections} with the difference,\ $A(t)-A$,\ constructed from $F_{A}$.\ In \cite{bourguignon1981stability},\ the test {connections} $A(t)=A+ti_{X}F^{+}_{A}$ {were} used,\ where $i_{X}$ denotes interior multiplication by the vector field $X$,\ and $X$ runs over a basis of Killing vector fields.\ In \cite{MS},\ Stern's proof of Theorem 1.1 extends the variational argument of Bourguignon,\ Lawson,\ and Simons \cite{bourguignon1981stability}.

The curvature is the only natural object from which to construct test variations,\ but we need a map from $2$-forms to $1$-forms in order to create test variations from the curvature $2$-form.\ {On a} $G_{2}$-manifold,\ there exists a covariant $3$-form $\phi^{(3)}$,\ {which induces} a natural map from $2$-forms to $1$-forms.

On a K\"{a}hler $3$-manifold with K\"{a}hler form $\w$ the curvature decomposes as
$$F_{A}=F^{2,0}_{A}+F^{1,1}_{A0}+\frac{1}{3}(\La_{\w}F_{A})\otimes\w+F^{0,2}_{A},$$
where $\La_{\w}$ denotes the adjoint of exterior multiplication by $\w$,\ and $F^{1,1}_{A0}=F^{1,1}_{A}-\frac{1}{3}\hat{F}_{A}\otimes\w$,\ we denote
$\hat{F}_{A}=\frac{1}{3}(\La_{\w}F_{A})$.\ We {call} a connection $A$ an $\w$-instanton,\ if $F_{A}$ satisfies
$$F_{A}=-\ast(\w\wedge F_{A}).$$
The K\"{a}hler {identities} (see for example \cite{RW} Theorem 3.16)
$$\w\wedge F_{A}=\ast(F^{0,2}_{A}+\frac{2}{3}\hat{F}_{A}\otimes\w-F^{1,1}_{A0}+F^{0,2}_{A}).$$
Then {A is} an $\w$-instanton over a K\"{a}hler $3$-fold if only if
$$F^{0,2}_{A}=0\quad and \quad \La_{\w}{F}_{A}=0.$$
By Donaldson-Uhlenbeck-Yau theorem \cite{DK1985,UY},\ the bundle must be polystable.\\
In the case of Calabi-Yau $3$-folds $CY^{3}$,\ we consider a Hermitian vector bundle $E\rightarrow CY^{3}$ over $(CY^{3},\w)$.\ The Riemannian product manifold $M:=CY^{3}\times{S^{1}}$ is {naturally} a real $7$-dimensional $G_{2}$-manifold (\cite{Jo}, 11.1.2).\ We {pull} back a connection $A$ on $E\rightarrow CY^{3}$ to $p_{1}^{\ast}{E}\rightarrow M$ via the canonical projection
$$p_{1}: CY^{3}\times S^{1}\rightarrow CY^{3}.$$
From \cite{He} Proposition 8,\ the canonical projection gives a one-to-one correspondence Hermitian-Yang-Mills connections on $E$ and $S^{1}$-invariant $G_{2}$-instantons on the pullback bundle $p^{\ast}_{1}E$.\ In Section 4,\ we observation that the pullback connection $p^{\ast}_{1}A$ such that $\psi_{p_{1}^{\ast}A}$ is harmonic is equivalent to $\xi_{A}$ is also harmonic.\ Then we show that energy minimizing connections on a Calabi-Yau $3$-fold subject to a similar condition are holomorphic.
\begin{thm}\label{T1}
Let $(CY^{3},\w,\Om)$ be a compact Calabi-Yau $3$-fold {which} has holonomy equal to $SU(3)$,\ let $A$ be a Hermitian,\ stable Yang-Mills connection on a Hermitian bundle $E$ over $CY^{3}$.\ Assume $\xi_{A}$ is $d_{A}$-harmonic,\ $\xi_{A}\triangleq\ast(F^{0,2}_{A}\wedge\Om)$,\ then $(E,\bar{\pa}_{A})$ is a holomorphic bundle.\ Further more,\ if $A$ is irreducible,\ then $A$ is {an} $\w$-instanton.
\end{thm}

\section{Preliminaries}
First,\ we recall some standard  notations and definitions.

Let $M$ be a complete Riemannian manifold and $E$ a principal $G$ bundle over $M$,\ with $G$ a compact Lie group.\ Let $\mathfrak{g}_{E}$ denote the adjoint bundle of $E$,\ endowed with a $G$-invariant inner product.\ Let $\La^{p}(M,\mathfrak{g}_{E})$ denote the smooth $p$-forms with values in $\mathfrak{g}_{E}$.\ Given a connection on
$E$,\ we denote by $\na_{A}$ the corresponding covariant derivative on $\La^{\ast}(M,\mathfrak{g}_{E})$ induced by $A$ and the Levi-Civita connection of $M$.\ Let
$d_{A}$ denote the exterior derivative associated to $\na_{A}$.\ For $\na_{A}$ and $d_{A}$,\ we have adjoint operators $\na^{\ast}_{A}$ and $d^{\ast}_{A}$.\ We also have {the Weitzenb\"{o}ck} formula (\cite{bourguignon1981stability},\ Theorem 3.2)
\begin{equation}\label{2.1}
(d_{A}d_{A}^{\ast}+d_{A}^{\ast}d_{A})\varphi=\na^{\ast}_{A}\na_{A}\varphi+\varphi\circ Ric+R^{A}(\varphi)
\end{equation}
where $\varphi\in\La^{1}(M,\mathfrak{g}_{E})$,\ $Ric$ is the Ricci tensor.

In a local orthonormal frame $(e_{1},\ldots,e_{n})$ of $TM$ ,\ the operators of $\varphi\circ Ric$ and $R^{A}(\varphi)$ are defined by Bourguignon and Lawson
\cite{bourguignon1981stability} as follows.
\begin{equation}\nonumber
\varphi\circ Ric(e_{i})=\sum_{j=1}^{n}R_{ij}\varphi_{j}.
\end{equation}
and
\begin{equation}\nonumber
R^{A}(\varphi)_{X}\equiv\sum_{j=1}^{n}[F_{A}(e_{j},X),\varphi_{j}].
\end{equation}
We are interested in minima of the Yang-Mills energy
\begin{equation}\nonumber
YM(A)=\|F_{A}\|^{2}_{L^{2}(M)},
\end{equation}
where $F_{A}$ denotes the curvature of $A$.\ Critical points of this energy satisfy the Yang-Mills equation
$$d^{\ast}_{A}F_{A}=0,$$
where $d^{\ast}_{A}$ denotes adjoint of $d_{A}$.\ In addition,\ all connections satisfy the Bianchi identity
$${d_{A}F_{A}=0.}$$
If $\psi\in\La^{1}(M,\mathfrak{g}_{E})$ then
\begin{equation}\label{1.1}
F_{A+\psi}=F_{A}+d_{A}\psi+\psi\wedge\psi.
\end{equation}
{Here $\a\wedge\b$ ($\a,\b\in\La^{1}(M,\mathfrak{g}_{E})$) denotes the wedge product of the two forms with the Lie bracket used to combine the values in $\mathfrak{g}$.\ In detail,\ One can see this in \cite{We} Appendix A.}
As a notional convenience,\ we will often use $L_{\w}$ to {denoted} exterior multiplication on the left by a form $\w$.\ Its adjoint is denote $\La_{\w}$.\
Thus
$$L_{\w}h:=\w\wedge h,\ and\ \langle f,L_{\w}h\rangle=\langle\La_{\w}f,h \rangle.$$
If $A$ minimizes the Yang-Mills energy,\ then of course it satisfies the inequality
\begin{equation}\label{p2}
\|F_{A}\|^{2}\leq\|F_{A+\psi}\|^{2}
\end{equation}
for all smooth compactly supported $\psi$.\ Replacing $\psi$ by $t\psi$ in (\ref{p2}),\ using (\ref{1.1}),\ and taking the limit as $t\rightarrow0$ leads
to the second variational inequality
\begin{equation}\label{p3}
0\leq\|d_{A}\psi\|^{2}+2\langle F_{A},\psi\wedge\psi\rangle
\end{equation}

\section{Yang-Mills connection and $G_{2}$-instanton}
{On} this section,\ we begin to prove the first main theorem of our article.\ At first,\ we introduce some background knowledge about $G_{2}$-manifold and {$G_{2}$-instantons}.

\subsection{Gauge theory in $G_{2}$-manifolds}
This section is devoted to the background language for the subsequent analytical investigation.\ The main references are \cite{RLB1987,RLB,He,Jo,Sa,MV}.

\begin{defin}({\cite{MV} Definition 3.3})
Let $M$ be a $7$-dimensional smooth manifold,\ and $\phi^{(3)}\in\La^{3}(M)$ a $3$-form.\ $(M,\phi^{(3)}$) is called a $G_{2}$-manifold if $\phi^{(3)}$ is non-degenerate and positive everywhere on $M$.\ The manifold $(M,g,\phi^{(3)})$ is called a holonomy $G_{2}$-manifold if $\phi^{(3)}$ is parallel with respect to the Levi-Civita connection associated with $g$.\ Further on,\ we shall consider only holonomy $G_{2}$-manifolds,\ and (abusing the language) omit the word holonomy.
\end{defin}
Under the action of $G_{2}$,\ the space $\La^{2}(M)$ splits into irreducible representations,\ as follows.
$$\La^{2}(M)=\La^{2}_{7}(M)\oplus\La_{14}^{2}(M)$$
These summands for $\La^{2}(M)$ can be characterized as follows:
\begin{equation}
\La^{2}_{7}(M)=\{\a\in\La^{2}(M)\mid\ast(\a\wedge\phi^{(3)})=2\a\}=\{\a\in\La^{2}(M)\mid\ast(\psi^{(4)}\wedge\ast(\psi^{(4)}\wedge\a))=3\a\},
\end{equation}
\begin{equation}
\La^{2}_{14}(M)=\{\a\in\La^{2}(M)\mid\ast(\a\wedge\phi^{(3)})=-\a\}=\{\a\in\La^{2}(M)\mid\a\wedge\psi^{(4)}=0\}.
\end{equation}
We define a projection operator
$$\Pi^{2}_{7}:\La^{2}(M)\rightarrow\La^{2}_{7}(M),$$
\begin{equation}\label{1}
\a\mapsto\Pi^{2}_{7}(\a)=\frac{1}{3}(\a+\ast(\phi^{(3)}\wedge\a))=\frac{1}{3}\ast(\psi^{(4)}\wedge{\ast}(\psi^{(4)}\wedge\a)).
\end{equation}
By the definition of $\La^{2}_{7}(M)$,\ it is easy to get
\begin{prop}\label{3.1}
If $\a\in\La^{2}(M,\mathfrak{g}_{E})$,\ then we have a point-wise identity
$$3|\Pi^{2}_{7}\a|^{2}=|\b|^{2},$$
where $\b=\ast(\a\wedge\psi^{(4)})$.
\end{prop}

Consider a vector bundle $E\rightarrow M$ over a compact $G_{2}$-manifold $(M,\phi^{(3)})$,\ the curvature $F_{A}$ of connection $A$ decomposes as:
$$F_{A}=F^{7}_{A}+F^{14}_{A},\quad F^{i}_{A}\in\La^{2}_{i}(M,\mathfrak{g}_{E}),\quad i=7,14.$$
The Yang-Mills functional is
\begin{equation}\label{G4}
YM(A):=\|F_{A}\|^{2}_{L^{2}(M)}=\|F^{7}_{A}\|^{2}_{L^{2}(M)}+\|F^{14}_{A}\|^{2}_{L^{2}(M)}.
\end{equation}
It is well-known that the values of $YM(A)$ can be related to a certain characteristic class of the bundle $E$,\ given by
$$\kappa(E):=-\int_{M}tr(F^{2}_{A})\wedge\phi^{(3)}.$$
Using the property $d\phi^{(3)}=0$,\ we {know} that the de-Rham class $[tr(F^{2}_{A})\wedge\phi^{(3)}]$ is independent of $A$,\ thus the integral is a topological invariant.\ From the {decomposition} of $F_{A}$,\ we have
$$\kappa(E)=-2\|F^{7}_{A}\|^{2}_{L^{2}(M)}+\|F^{14}_{A}\|^{2}_{L^{2}(M)},$$
and combing with (\ref{G4}) we get
$$YM(A)=3\|F^{7}_{A}\|^{2}_{L^{2}(M)}+\kappa(E).$$
Hence $YM(A)$ attains its absolute minimum at a connection whose curvature lies either in $\La^{2}_{7}$ or in $\La^{2}_{14}$.\ We call a connection $A$ is a $G_{2}$-instanton,\ if $F_{A}$ satisfies
$$F_{A}\wedge\psi^{(4)}=0,$$
or,\ equivalently
$$F_{A}+\ast(F_{A}\wedge\phi^{(3)})=0.$$

\subsection{Stable Yang-Mills connections and $G_{2}$-instantons}
In this section,\ we begin to prove our main theorem,\ we return to consider the Yang-Mills connection over $G_{2}$-manifold.\ At first,\ we define $\psi_{A}\in{\La^{1}(M,\mathfrak{g}_{E}})$ such that
\begin{equation}\label{3.7}
\ast(\psi^{(4)}\wedge\psi_{A})=F_{A}^{7}.
\end{equation}
Using the identity, (see \cite{RLB})
$$\ast\big{(}\ast(\a\wedge\ast\phi^{(3)})\wedge\ast\phi^{(3)}\big{)}=3\a, \quad\forall\a\in\La^{1}(M),$$
hence,\ we have
$$\psi_{A}=\frac{1}{3}\big{(}\ast(F^{7}_{A}\wedge\psi^{(4)})\big{)}.$$
\begin{lem}\label{3.9}
Let $A$ be a connection on a complete $G_{2}$-manifold,\ then $$d^{\ast}_{A}\psi_{A}=0,$$
where $\psi_{A}$ is defined as (\ref{3.7}).\ Furthermore,\ if $A$ is a Yang-Mills connection,\ $\psi_{A}$ also satisfies
\begin{equation}\label{3.10}
\Pi^{2}_{7}(d_{A}\psi_{A})=0.
\end{equation}
\end{lem}
\begin{proof}
First,\ from the Bianchi identity $d_{A}F_{A}=0$ and the fact $d\psi^{(4)}=0$,\ we have
\begin{equation}\nonumber
0=d_{A}(F_{A}\wedge\psi^{(4)})=d_{A}(F^{7}_{A}\wedge\psi^{(4)})=3d_{A}\ast\psi_{A}.
\end{equation}
Hence we obtain $d_{A}^{\ast}\psi_{A}=0$.\\
Further more,\ if $A$ is a Yang-Mills connection,\ using Bianchi identity again,\ we have
$$d_{A}^{\ast}F^{7}_{A}=\frac{1}{3}\ast d_{A}(F_{A}\wedge\phi^{(3)})=0.$$
We {applying the} operator $d^{\ast}_{A}$ to (\ref{3.7}) each side,\ then we get
\begin{equation}\label{3.3}
\ast(d_{A}\psi_{A}\wedge\psi^{(4)})=0
\end{equation}
Hence form the Proposition \ref{3.1} and (\ref{3.3}),\ we have
$$\Pi^{2}_{7}(d_{A}\psi_{A})=0.$$
\end{proof}
Now,\ we {define} a connection on a $G_{2}$-manifold has a harmonic curvature.\ {The definition is inspired by Itoh's article \cite{MI}.}
\begin{defin}
A connection $A$ on a complete $G_{2}$-manifold is said to {has} harmonic curvature if $\psi_{A}$ is $d_{A}$-harmonic,\ $\psi_{A}$ is defined as in (\ref{3.7}),\ i.e,\ $d_{A}\psi_{A}=d^{\ast}_{A}\psi_{A}=0$.
\end{defin}
In this article,\ we consider a Yang-Mills connection $A$ {which} has harmonic curvature on $E$ over a $G_{2}$-manifold.\ {At first,\ we prove a useful lemma as follow.}

\begin{lem}\label{G3}
Let $M$ be a compact $G_{2}$-manifold,\ let $A$ be a stable Yang-Mills connection on a bundle $E$ over $M$ with compact,\ semi-simple Lie group.\ Assume $\eta\in\La^{1}(M,\mathfrak{g}_{E})$ is $d_{A}$-harmonic,\ then
$$[\ast F^{7}_{A},\eta]=0.$$
\end{lem}
\begin{proof}
At first,\ we consider the variation $A+t\eta$,\ $\eta\in\La^{1}(M,\mathfrak{g}_{E})$,\ hence we have
\begin{equation}\label{3.11}
\|F_{A+t\eta}\|^{2}=3\|F^{7}_{A+t\eta}\|^{2}+\kappa(E).
\end{equation}
A direct {calculation shows}
\begin{equation}\nonumber
F^{7}_{A+t\eta}=F^{7}_{A}+t\Pi^{2}_{7}(d_{A}\eta)+t^{2}\Pi^{2}_{7}(\eta\wedge \eta),
\end{equation}
here {the definition of $\eta\wedge\eta$ is the same as the notation in Section 2.}\ Since $A$ is a stable Yang-Mills connection,\ we have
$$\frac{d^{2}}{dt^{2}}\big{|}_{t=0}\|F_{A+t\eta}\|^{2}_{L^{2}(M)}\geq0,$$
for any $\eta\in\La^{1}(M,\mathfrak{g}_{E})$, i.e.,
\begin{equation}\label{3.15}
\|\Pi^{7}_{2}(d_{A}\eta)\|^{2}+2\langle F^{7}_{A},\eta\wedge\eta\rangle\geq0.
\end{equation}
here,\ we used the identity
$$\langle F^{7}_{A},\eta\wedge \eta\rangle=\langle F^{7}_{A},\Pi^{2}_{7}(\eta\wedge\eta)\rangle.$$
Now,\ we assume $\eta\in\La^{1}(M,\mathfrak{g}_{E})$ is $d_{A}$-harmonic,\ i.e.,\ $d^{\ast}_{A}\eta=d_{A}\eta=0$,\ then the Weitzenb\"{o}ck formula gives
\begin{equation}\label{3.6}
\na^{\ast}_{A}\na_{A}\eta+R^{A}(\eta)=0,
\end{equation}
here we have used the vanishing of the Ricci curvature on $G_{2}$-manifold.\ Hence the $L^{2}$-inner product of (\ref{3.6}) with $\eta$,\ we obtain
\begin{equation}\label{Y25}
\|\na_{A}\eta\|^{2}+\langle F_{A},[\eta,\eta]\rangle=0.
\end{equation}
Since $d_{A}\eta=0$,\ thus $d_{A}d_{A}\eta=[F_{A},\eta]=0$,\ then
\begin{equation}\label{3.2}
\ast([F_{A},\eta]\wedge\phi^{(3)})=0.
\end{equation}
The $L^{2}$-inner product of (\ref{3.2}) with $\eta$
\begin{equation}\nonumber
\begin{split}
0&=\langle\ast([F_{A},\eta]\wedge\phi^{(3)}),\eta\rangle=-\langle\ast([F_{A}\wedge\phi^{(3)},\eta]),\eta\rangle\\
&=-\langle\ast([2\ast F^{7}_{A}-\ast F^{14}_{A},\eta]),\eta\rangle\\
&=-2\langle F^{7}_{A},[\eta,\eta]\rangle+\langle F^{14}_{A},[\eta,\eta]\rangle.\\
\end{split}
\end{equation}
Then,\ we have a identity:
\begin{equation}\label{G1}
2\langle F^{7}_{A},[\eta,\eta]\rangle=\langle F^{14}_{A},[\eta,\eta]\rangle.
\end{equation}
From (\ref{Y25}) and (\ref{G1}),\ we get
\begin{equation}\label{G2}
\begin{split}
\|\na_{A}\eta\|^{2}&=-\langle F_{A},[\eta,\eta]\rangle=-\langle F^{7}_{A},[\eta,\eta]\rangle-\langle F^{14}_{A},[\eta,\eta]\rangle\\
&=-3\langle F^{7}_{A},[\eta,\eta]\rangle.\\
\end{split}
\end{equation}
From (\ref{3.15}) and $d_{A}\eta=0$,\ we have
\begin{equation}\label{3.8}
\langle F^{7}_{A},[\eta,\eta]\rangle\geq0.
\end{equation}
From (\ref{G2}) and (\ref{3.8}),\ we obtain
$$\langle F^{7}_{A},\eta\wedge\eta\rangle=0$$
Now,\ we consider the variant $t\eta+t^{\frac{3}{2}}\w$,\ from (\ref{3.15}) we can get
\begin{equation}\nonumber
\begin{split}
0&\leq\|\Pi^{7}_{2}\big{(}d_{A}(t\eta+t^{\frac{3}{2}}\w)\big{)}\|^{2}_{L^{2}(M)}
+2\langle F^{7}_{A},(t\eta+t^{\frac{3}{2}}\w)\wedge(t\eta+t^{\frac{3}{2}}\w)\rangle\\
&=t^{3}\|\Pi^{7}_{2}(d_{A}\w)\|^{2}+2t^{\frac{5}{2}}\langle F^{7}_{A},[\eta,\w]\rangle+t^{3}\langle F^{7}_{A},\w\wedge\w\rangle\\
\end{split}
\end{equation}
Hence
$$0\leq t^{\frac{1}{2}}\|\Pi^{7}_{2}(d_{A}\w)\|^{2}+2\langle F^{7}_{A},[\eta,\w]\rangle
+t^{\frac{1}{2}}\langle F^{7}_{A},\w\wedge\w\rangle.$$
Taking $t\rightarrow0$,\ then we have
\begin{equation}\label{G5}
0\leq\langle F^{7}_{A},[\eta\wedge\w]\rangle=\langle\ast[\ast{F^{7}_{A}},\eta],\w\rangle,\ \w\in\La^{1}(M,\mathfrak{g}_{E}).
\end{equation}
We can choose $\w=-\ast[\ast{F^{7}_{A}},\eta]$,\ then the (\ref{G5}) yields
$$0=\|[\ast{F^{7}_{A}},\eta]\|^{2}_{L^{2}(M)}.$$
Hence,\ we have
\begin{equation}\label{3.16}
[\ast{F^{7}_{A}},\eta]=0.
\end{equation}
\end{proof}
\begin{rem}
In process of proving Lemma \ref{G3},\ we need the fact $\kappa(E):=-\int_{M}tr(F^{2}_{A}\wedge\phi^{(3)})$ is a topological invariant.\ Unfortunately,\ we can't {extend $\kappa(E)$ as a} topological invariant to a non-compact manifold.\ So the compactness of $G_{2}$-manifold in Lemma \ref{G3} is indispensable.\
\end{rem}
\begin{lem}
Let $A$ be a stable Yang-Mills connection {which} has harmonic curvature over a compact $G_{2}$-manifold,\ then the components of $\psi_{A}$ thus generates an abelian subalgebra of $\mathfrak{g}_{E}$,\ here $\psi_{A}$ is defined as (\ref{3.7}).
\end{lem}
\begin{proof}
Since $F^{7}_{A}=\ast(\psi_{A}\wedge\psi^{(4)})$, then $\ast{F^{7}_{A}}=\psi_{A}\wedge\psi^{(4)}$.\ Hence,
\begin{equation}\nonumber
\begin{split}
[\ast{F^{7}_{A}},\psi_{A}]&=(\psi_{A}\wedge\psi^{(4)})\wedge\psi_{A}+\psi_{A}\wedge(\psi_{A}\wedge\psi^{(4)})\\
&=2(\psi_{A}\wedge\psi_{A})\wedge\psi^{(4)}.\\
\end{split}
\end{equation}
Since $\psi_{A}$ is harmonic,\ from Lemma \ref{G3},\ we obtain
\begin{equation}\label{3.17}
0=[\ast{F^{7}_{A}},\psi_{A}]=2(\psi_{A}\wedge\psi_{A})\wedge\psi^{(4)}.
\end{equation}
i.e.
$$\psi_{A}\wedge\psi_{A}\in\La^{2}_{14}(M,\mathfrak{g}_{E}).$$
Choosing $A(t)=A+t\psi_{A}$,\ then we get
$$F^{7}_{A+t\psi_{A}}=F^{7}_{A}+t\Pi^{2}_{7}(d_{A}\psi_{A})+t^{2}\Pi^{2}_{7}(\psi_{A}\wedge\psi_{A})=F^{7}_{A}.$$
Then,\ we have
$$YM(A+t\psi_{A})=YM(A).$$
The preceding identity implies that the quartic polynomial in $t$,\ $YM(A+t\psi_{A})$ is in fact constant.\  In particular,\ the  coefficient of $t^{4}$ and this vanish implies,
\begin{equation}\label{G20}
0=\psi_{A}\wedge\psi_{A}.
\end{equation}
\end{proof}
\textbf{Proof Main Theorem}.\ The proof is similar to Theorem 6.21 in \cite{MS}.\\
{The Weitzenb\"{o}ck formula gives}
$$0=\|d_{A}\psi_{A}\|^{2}+\|d^{\ast}_{A}\psi_{A}\|^{2}=\|\na_{A}\psi_{A}\|^{2}.$$
Here we used the vanishing of the Ricci curvature on $G_{2}$-manifolds and (\ref{G20}).\ Hence
$$\na_{A}\psi_{A}=0.$$
Let $R_{ij}dx^{i}\wedge dx^{j}$ denote the Riemann curvature tensor viewed as an $ad(T^{\ast}M)$ valued $2$-form,\ The vanishing of $\na_{A}\psi_{A}$ implies
$$0=[\na_{i},\na_{j}]\psi_{A}=ad((F_{ij})+R_{ij})\psi_{A}$$
for all $i,j$.\ Because $\psi_{A}$ takes values in an abelian subalgebra of $\mathfrak{g}_{E}$,\ $[F_{ij},\psi_{A}]\perp R_{ij}\psi_{A}$.\ Hence $R_{ij}\psi_{A}=0$,\ and the components of $\psi_{A}$ are in the kernel of the Riemannian curvature operator.\ This reduces the Riemannian holonomy group,\ unless $\psi_{A}=0$ which implies $F^{7}_{A}=0$ (from Proposition \ref{3.1}).\ Thus,\ we have the dichotomy:\ $\psi_{A}\neq0$ implies a reduction of the holonomy of $M$,\ and $\psi_{A}$ implies the connection $A$ is $G_{2}$-instanton.

\section{Calabi-Yau $3$-fold}

In this section,\ we begin to prove the second theorem of our article.\ At first,\ we consider the Yang-Mills connection $A$ on a bundle $E$ over a K\"{a}hler manifold.\ The curvature splits into $F_{A}=F^{2,0}_{A}+F_{A}^{1,1}+F^{0,2}_{A}$,\ where $F^{p,q}_{A}$ is the $(p,q)$-component.\ Then we get from the Bianchi identity
\begin{equation}\nonumber
\begin{split}
&\pa_{A}F^{2,0}_{A}=\bar{\pa}_{A}F^{0,2}_{A}=0\\
&\bar{\pa}_{A}F^{2,0}_{A}+\pa_{A}F^{1,1}_{A}=\pa_{A}F^{0,2}_{A}+\bar{\pa}_{A}F^{1,1}_{A}=0.\\
\end{split}
\end{equation}
Decompose the curvature,\ $F_{A}$,\ as
$$ F_{A}=F^{2,0}_{A}+F^{1,1}_{A0}+\frac{1}{n}\hat{F}_{A}\otimes\w+F^{0,2}_{A}$$
where $\hat{F}_{A}:=\La_{\w}F_{A}$.
\begin{prop}\label{Y4}(\cite{HT2} Proposition 2.1)
Let $A$ be a Yang-Mills connection on a bundle $E$ over a K\"{a}hler $n$-fold,\ then
\begin{equation}\nonumber
\begin{split}
&(1)\ 2\bar{\pa}^{\ast}_{A}F^{0,2}_{A}=\sqrt{-1}\bar{\pa}_{A}\hat{F}_{A},\\
&(2)\ 2\pa_{A}^{\ast}F^{2,0}_{A}=-\sqrt{-1}\pa_{A}\hat{F}_{A}.\\
\end{split}
\end{equation}
\end{prop}
The argument {which follow is the same as in} \cite{He} Section 1.3.\ Let $CY^{3}$ be a Calabi-Yau $3$-fold,\ we will see that the Cartesian product $M=CY^{3}\times S^{1}$ is a naturally a real $7$-dimensional $G_{2}$-manifold.\ Starting with the K\"{a}hler form $\w\in\La^{1,1}(CY^{3})$ and holomorphic volume form $\Om$ on $CY^{3}$ (\cite{GHJ}),\ define
\begin{equation}\label{D1}
\begin{split}
&\phi^{(3)}=\w\wedge d\theta+Im\Om,\\
&\psi^{(4)}=\frac{1}{2}\w^{2}-Re\Om\wedge d\theta.\\
\end{split}
\end{equation}
Here $\theta$ is the coordinate $1$-form on $S^{1}$,\ and the Hodge {star} on $M$ is given by the product of the K\"{a}hler metric on $CY^{3}$ and the standard flat metric on $S^{1}$.\ Now,\ a connection $A$ on $E\rightarrow CY^{3}$ pulls back to $p_{1}^{\ast}E\rightarrow CY^{3}\times S^{1}$ via the canonical projection:
$$p_{1}:CY^{3}\times S^{1}\rightarrow CY^{3},$$
and so do the forms $\w$ and $\Om$ (for simplicity we keep the same notation for objects on $CY^{3}$ and their pullbacks to $M$).\ The Proposition 8 in \cite{He} shows that a connection on $E$ whose curvature $F_{A}$ is a $\w$-instanton,\ then the pullback of $A$ to the $G_{2}$-manifold $CY^{3}\times S^{1}$ is a $G_{2}$-instanton.\\
In this article,\ we consider $E \rightarrow CY^{3}$ is a Hermitian vector bundle and $A$ is the Hermitian connection on $E$.\ We choose a {local} trivialization such that $E$ is isomorphic to the trivial bundle.\ Then $d_{A}:=d+A$ with $\bar{A}^{T}=-A$.\ For its curvature $F_{A}=d_{A}+A\wedge A$ one obtains
$$\bar{F}^{T}_{A}=-F_{A},$$
then we have
$$F^{0,2}_{A}=-(\bar{F}^{0,2}_{A})^{T}\ and\ \La_{w}F^{1,1}_{A}=-\overline{\La_{\w}F^{1,1}_{A}}^{T}.$$
We define $\xi_{A}\in\La^{0,1}(M,\mathfrak{g}^{\C}_{E})$,\ such that
\begin{equation}\label{CY3}
\ast(\xi_{A}\wedge\Om)=F^{0,2}_{A}.
\end{equation}
Then,\ we have a useful
\begin{prop}\label{CY1}
Let $A$ be a Hermitian connection on a Hermitian vector bundle $E$ over a complete Calabi-Yau $3$-fold.\ The following conditions are equivalent:\\
(1)\ $\bar{\pa}_{A}(\La_{\w}F_{A})=0$;\\
(2)\ $\xi_{A}$ is harmonic.\\
where $\psi_{A}$ is defined as (\ref{CY3}).
\end{prop}
\begin{proof}
In a local special unitary frame,\
$$F^{0,2}_{A}=F^{0,2}_{23}d\bar{z}^{2}\wedge d\bar{z}^{3}+F^{0,2}_{31}d\bar{z}^{3}\wedge d\bar{z}^{1}+F^{0,2}_{12}d\bar{z}^{1}\wedge d\bar{z}^{2},$$
hence from (\ref{CY3})
$$ \xi_{A}=-(F^{2,0}_{23}d\bar{z}^{1}+F^{2,0}_{31}d\bar{z}^{2}+F^{2,0}_{12}d\bar{z}^{3}),$$
{By a directly calculation},\ we have
$$\ast(F^{0,2}_{A}\wedge\Om)=\xi_{A}.$$
The Bianchi identity implies $\bar{\pa}_{A}F^{0,2}_{A}=0$,\ which is equivalent to
\begin{equation}
\bar{\pa}^{\ast}_{A}\xi_{A}=\ast(\bar{\pa}_{A}F^{0,2}_{A}\wedge\Om)=0.
\end{equation}
Applying $\bar{\pa}^{\ast}_{A}$ to each side of (\ref{CY3}) gives
$$-\ast(\bar{\pa}_{A}\xi_{A}\wedge\Om)=\bar{\pa}_{A}^{\ast}F^{0,2}_{A}.$$
Hence $\bar{\pa}_{A}^{\ast}F^{0,2}_{A}=0$ i.e., $\bar{\pa}_{A}(\La_{w}F_{A})=0$ (see Proposition \ref{Y4}),\ is equivalent to $\bar{\pa}_{A}\xi_{A}=0$.
\end{proof}
\begin{prop}\label{D3}
Let $CY^{3}$ be a compact Calabi-Yau $3$-fold,\ let $E$ be a Hermitian vector bundle over $X$,\ let $A$ be a Hermitian connection on $E$.\ Assume $A$ is a stable Yang-Mills connection,\ then the pullback of $A$ to the $G_{2}$-manifold $CY^{3}\times S^{1}$ is also a stable Yang-Mills connection.\ {Furthermore},\  we denote the pullback of $A$ {by} $\tilde{A}$,\ then the follow conditions are equivalent\\
(1) $\xi_{A}$ is $d_{A}$-harmonic,\ $\xi_{A}$ is {define as in} (\ref{CY3}),\\
(2) $\psi_{\tilde{A}}$ is $D_{\tilde{A}}$-harmonic,\ $\psi_{\tilde{A}}$ is {define as in} (\ref{3.7}).
\end{prop}
\begin{proof}
Since $A$ is a Yang-Mills connection\ i.e.,\ $d_{A}^{\ast}F_{A}=0$,\ hence
$$D_{\tilde{A}}\ast F_{\tilde{A}}=D_{\tilde{A}}(d\theta\wedge\ast_{CY^{3}}F_{\tilde{A}})=-d\theta\wedge p_{1}^{\ast}(d_{A}\ast_{CY^{3}}F_{A})=0.$$
Here {${D_{\tilde{A}}}$ denote the covariant derivative with respect to $\tilde{A}$ and $d_{A}$ is the covariant derivative with respect to $A$},\ $\ast_{CY^{3}}$ denotes the Hodge $\ast$-operator on $CY^{3}$.\ If $A$ is a stable Yang-Mills connection over $CY^{3}$,\ then for $\a\in\La^{1}(CY^{3},\mathfrak{g}_{E})$,\ we have:
$$\int_{CY^{3}}|d_{A}\a|^{2}+2\langle F_{A},\a\wedge\a\rangle\geq0.$$
Hence,\ for $\tilde{\eta}\in\La^{1}(CY^{3}\times S^{1},\mathfrak{g}_{E})$ (we denote $\tilde{\eta}(\theta,x):=\eta_{0}d\theta+\eta(\theta,x)$),\ we have
\begin{equation}\nonumber
\begin{split}
&\quad\int_{CY^{3}\times S^{1}}|D_{\tilde{A}}\tilde{\eta}(\theta,x)|^{2}+2\langle F_{\tilde{A}},\tilde{\eta}(\theta,x)\wedge\tilde{\eta}(\theta,x)\rangle\\
&=\int_{S^{1}}\big{(}\int_{CY^{3}}|d_{\tilde{A}}{\eta}(\theta,x)|^{2}+2\langle F_{\tilde{A}},\eta(\theta,x)\wedge\eta(\theta,x)\rangle\big{)}d\theta+\int_{CY^{3}\times S^{1}}|d_{\tilde{A}}\eta_{0}\wedge d\theta|^{2}\geq0,\\
\end{split}
\end{equation}
then the pullback of $A$ to the $G_{2}$-manifold $CY^{3}\times S^{1}$ is also a stable Yang-Mills connection.\\
If $A$ is a Yang-Mills connection and $\xi_{A}$ is $d_{A}$-harmonic,\ i.e.
$$d_{A}(F^{0,2}_{A}\wedge\Om)=d^{\ast}_{A}(F^{0,2}_{A}\wedge\Om)=0,$$
here $\Om$ is holomorphic volume on $CY^{3}$.\ Then the pullback of $A$ to the $G_{2}$-manifold $CY^{3}\times S^{1}$ satisfies
\begin{equation}\nonumber
\begin{split}
3\psi_{\tilde{A}}&=\ast(F_{\tilde{A}}\wedge\psi^{(4)})=\ast\big{(}F_{\tilde{A}}\wedge(\frac{1}{2}\w^{2}-Re\Om\wedge d\theta)\big{)}\\
&=d\theta\wedge\ast_{CY^{3}}(F_{\tilde{A}}\wedge\frac{1}{2}\w^{2})
+\ast_{CY^{3}}\big{(}(F^{0,2}_{\tilde{A}}+F^{2,0}_{\tilde{A}})\wedge\frac{1}{2}(\Om+\bar{\Om})\big{)}.\\
&=(\La_{\w}F_{\tilde{A}})d\theta+\frac{1}{2}(\xi_{\tilde{A}}+\xi_{\tilde{A}}^{\dagger}).
\end{split}
\end{equation}
here $\xi_{\tilde{A}}\triangleq\ast(F^{0,2}_{A}\wedge\Om)$ and $\xi_{\tilde{A}}^{\dagger}\triangleq\ast(F^{0,2}_{A}\wedge\bar{\Om})$.\ In a local coordinate,\  $\xi_{\tilde{A}}^{\dagger}=-\bar{\xi}_{A}^{T}$.\ Hence,\ we have
\begin{equation}\nonumber
3D_{\tilde{A}}\psi_{\tilde{A}}=d_{\tilde{A}}(\La_{\w}F_{\tilde{A}})\wedge d\theta+\frac{1}{2}d_{\tilde{A}}(\xi_{\tilde{A}}+\xi_{\tilde{A}}^{\dagger})=0.
\end{equation}
here,\ we used the identity $d_{\tilde{A}}\xi_{\tilde{A}}^{\dagger}=-\overline{d_{\tilde{A}}\xi_{\tilde{A}}}^{T}=0$ (in a local coordinate) and Proposition \ref{CY1}.\ Hence if $\xi_{A}$ is $d_{A}$-harmonic,\ $\psi_{\tilde{A}}$ is $D_{\tilde{A}}$-harmonic.\\
If $\psi_{\tilde{A}}$ is $D_{\tilde{A}}$-harmonic,\ then we get
$$0=d_{\tilde{A}}(\xi_{\tilde{A}}+{\xi}^{\dagger}_{\tilde{A}})=\bar{\pa}_{A}\xi_{A}+\pa_{A}{\xi}^{\dagger}_{\tilde{A}}
+(\pa_{A}\xi_{A}+\bar{\pa}_{A}{\xi}^{\dagger}_{\tilde{A}}),$$
and
$$d_{\tilde{A}}\La_{\w}F_{\tilde{A}}=0,$$
hence,\ we have
$$0=\bar{\pa}_{\tilde{A}}\xi_{\tilde{A}},\ 0=\pa_{\tilde{A}}\xi_{\tilde{A}}+\bar{\pa}_{\tilde{A}}{\xi}^{\dagger}_{\tilde{A}},$$
and
$$0=\pa_{\tilde{A}}\La_{\w}F_{\tilde{A}}=\bar{\pa}_{\tilde{A}}\La_{\w}F_{\tilde{A}}.$$
Then,\ the connection $A$ on $E$ over Calabi-Yau $3$-fold and $\xi_{A}$ satisfy
\begin{equation}\label{D5}
0=\bar{\pa}_{A}\xi_{A}\ and\ 0=\pa_{A}\xi_{A}+\bar{\pa}_{A}{\xi}^{\dagger}_{A}
\end{equation}
We observed {an} identity,
$$\pa_{A}\pa_{A}\xi_{A}=[F^{2,0}_{A},\xi_{A}]=0.$$
By K\"{a}hler {identities},\ we have
$$\sqrt{-1}\bar{\pa}^{\ast}_{A}\pa_{A}\xi_{A}=\La_{w}(\pa_{A}\pa_{A}\xi_{A})-\pa_{A}(\La_{w}\pa_{A}\xi_{A})
=-\sqrt{-1}\pa_{A}\bar{\pa}^{\ast}_{A}\xi_{A}=0.$$
Then,\ we get
\begin{equation}\nonumber
\begin{split}
0&=\|\pa_{A}\xi_{A}+\bar{\pa}_{A}{\xi}_{A}^{\dagger}\|^{2}
=\|\pa_{A}\xi_{A}\|^{2}+\|\bar{\pa}_{A}{\xi}_{A}^{\dagger}\|^{2}+2\langle\pa_{A}\xi_{A},\bar{\pa}_{A}{\xi}_{A}^{\dagger}\rangle\\
&=\|\pa_{A}\xi_{A}\|^{2}+\|\bar{\pa}_{A}{\xi}_{A}^{\dagger}\|^{2}+2\langle\bar{\pa}_{A}^{\ast}\pa_{A}\xi_{A},{\xi}_{A}^{\dagger}\rangle\\
&=\|\pa_{A}\xi_{A}\|^{2}+\|\bar{\pa}_{A}{\xi}_{A}^{\dagger}\|^{2},\\
\end{split}
\end{equation}
hence,\ we get
\begin{equation}
\pa_{A}\xi_{A}=0.
\end{equation}
Using K\"{a}hler {identities} again,
\begin{equation}\label{3.18}
\bar{\pa}^{\ast}_{A}\xi_{A}=-\sqrt{-1}[\La_{\w},\pa_{A}]\xi_{A}=0.
\end{equation}
Combing the preceding identities (\ref{D5})--(\ref{3.18}) yields,
$$d_{A}\xi_{A}=d^{\ast}_{A}\xi_{A}=0.$$
\end{proof}
\textbf{Proof Theorem \ref{T1}}.\ Since $A$ is a stable Yang-Mills Hermitian connection and $\xi_{A}$ is $d_{A}$-harmonic,\ then the pullback connection $A$ (for simplicity we keep the same notation for pull back connection on $M:=CY^{3}\times S^{1}$) is also a stable Yang-Mills connection over $G_{2}$-manifold $M$ and $\psi_{A}$ is $D_{A}$-harmonic.\ Then,\ we can get
$$0=\psi_{A}\wedge\psi_{A}=
\big{(}(\La_{\w}F_{A})d\theta+\frac{1}{2}(\xi_{A}+{\xi}^{\dagger}_{A})\big{)}\wedge
\big{(}(\La_{\w}F_{A})d\theta+\frac{1}{2}(\xi_{A}+{\xi}^{\dagger}_{A})\big{)},$$
Hence
\begin{equation}\label{CY8}
0=[\La_{\w}F_{A},\xi_{A}]=[\xi_{A},\xi_{A}]=[\xi_{A},{\xi}^{\dagger}_{A}].
\end{equation}
The Weitzenb\"{o}ck formula now gives
$$0=\|d_{A}\xi_{A}\|^{2}+\|d^{\ast}_{A}\xi_{A}\|^{2}=\|\na_{A}\xi_{A}\|^{2}.$$
Here we used the vanishing of the Ricci curvature on Calabi-Yau manifolds and third identity of (\ref{CY8}).\ Hence
$$\na_{A}\xi_{A}=0.$$
Let $R_{ij}dx^{i}\wedge dx^{j}$ denote the Riemann curvature tensor viewed as an $ad(T^{\ast}M)$ valued $2$-form,\ The vanishing of $\na_{A}\xi_{A}$ implies
$$0=[\na_{i},\na_{j}]\xi_{A}=ad((F_{ij})+R_{ij})\xi_{A}$$
for all $i,j$.\ Because $\xi_{A}$ takes values in an abelian subalgebra of $\mathfrak{g}_{E}$,\ $[F_{ij},\xi_{A}]\perp R_{ij}\xi_{A}$.\ Hence $R_{ij}\xi_{A}=0$,\ and the components of $\xi_{A}$ are in the kernel of the Riemannian curvature operator.\ This reduces the Riemannian holonomy group,\ unless $\xi_{A}=0$ which implies $F^{0,2}_{A}=0$ .\ Thus,\ we have the dichotomy:\ $\xi_{A}\neq0$ implies a reduction of the holonomy of $CY^{3}$,\ and $\xi_{A}$ implies the bundle is holomorphic.\ Furthermore,\ if $A$ is a irreducible connection i.e., $ {\ker{d_{A}}|_{\Om^{0}(\mathfrak{g}_{E})}=0}$,\ since $d_{A}\La_{\w}F_{A}=0$,\ then $\La_{\w}F_{A}=0$.

\section*{Acknowledgment}
I would like to thank my supervisor Professor Sen Hu for providing numerous ideas during the course of stimulating exchanges.\ {I would like to thank the anonymous referee for a careful reading of my manuscript and helpful comments.}\ I also would like to thank Zhi Hu for further discussions about this work.\ I also would like to thank Professor Mark Stern for kind comments regarding this and {its} companion article \cite{MS}.\ This work is partially supported by Wu Wen-Tsun Key Laboratory of Mathematics of Chinese Academy of Sciences at USTC.

\end{document}